\documentclass[letterpaper, 10 pt, conference]{ieeeconf}  

\IEEEoverridecommandlockouts                              

\usepackage{amsmath} 
\usepackage{amssymb}  
\usepackage{bm}
\usepackage{algorithm}
\usepackage{algpseudocode}
\usepackage{mathtools}
\usepackage{nicefrac}
\usepackage{caption}
\usepackage{xcolor}
\usepackage{cancel}
\definecolor{deepred}{RGB}{210, 0, 0}

\newtheorem{definition}{Definition}
\newtheorem{problem}{Problem}

\newtheorem{proposition}{Proposition}
\newtheorem{remark}{Remark}

\newtheorem{assumption}{Assumption}

\title{\LARGE \bf
Analytic Non-Gaussian Confidence Boundary Method for Chance-Constrained Trajectory Control}

\author{Ethan R. Burnett$^{1}$ and Spencer Boone$^{2}$
\thanks{$^{1}$Ethan R. Burnett is affiliated with the Department of Aerospace Engineering and Engineering Mechanics, University of Texas at Austin, Austin, TX 78712, USA {\tt\small ethan.burnett@utexas.edu}}
\thanks{$^{2}$Spencer Boone is affiliated with ISAE SUPAERO, Toulouse, France {\tt\small spencer.boone@isae-supaero.fr}}
}

\begin{document}

\maketitle
\thispagestyle{empty}
\pagestyle{empty}

\begin{abstract}
Standard chance constrained control algorithms typically rely on the assumption that uncertainties in vehicle states obey Gaussian statistics. Highly nonlinear systems tend to disrupt Gaussianity, challenging standard chance-constrained control methods. This paper develops a non-Gaussian confidence boundary parameterization technique for such cases where the problem departs appreciably from the Gaussian assumption. The approach is to consider the true confidence boundary as a perturbation of the one predicted from covariance, deriving perturbed boundary geometry from computed higher-order statistical moments. Applying this technique to so-called ``banana-shaped distributions" (found e.g. in orbital mechanics problems) enables a simple parameterization of the confidence boundary using the skew and kurtosis tensors. The method is then applied to an impulsive stochastic spacecraft maneuver targeting problem in two-body dynamics. An algorithmic implementation outperforms a standard linear covariance-based approach in computing control parameters satisfying certain probabilistic bounds on the non-Gaussian distribution.
\end{abstract}

\section{Introduction}
Any practical control problem is characterized by some degree of uncertainty in the system states. Chance-constrained optimal control algorithms seek to achieve specific control goals while establishing probabilistic guarantees on the performance of an uncertain system. The system is often a vehicle, for which path constraints can be paramount. Most existing algorithms assume that system state uncertainties obey Gaussian statistics (e.g.,~\cite{OguriMcMahoIEEE2019}). Much of the work in chance-constrained trajectory control has focused on terrestrial robotic path-planning
~\cite{BlackmoreIEEE2011,RiskAlloc2008}. For such applications, the dynamics are often sufficiently linear or the measurement update frequency is generally high enough that the Gaussian assumption holds.

Control problems encountered in spaceflight differ from common terrestrial applications in several important ways. They are characterized by a potential for long periods without state measurements or corrective control maneuvers, and spacecraft are subject to the action of nonlinear dynamics in a chaotic environment \cite{Koon:2006rf}. Furthermore, the computational capabilities on a spacecraft are greatly reduced compared to terrestrial platforms, due to power limitations and design for a high-radiation environment. For safe long time-horizon maneuvers to be planned, spacecraft operators must accept the resulting non-Gaussian statistical distribution, and chance constraints must be enforced on the non-Gaussian distribution itself. To this end, Ref.~\cite{boone_cdc} applied a chance constrained approach with a Gaussian Mixture Model representation of the non-Gaussian distribution. The approach outperformed a linearized covariance-based approach, but it was numerically intensive. For rapid mission analysis or onboard implementation, more computationally efficient methods are needed. 

Some adjacent studies from robotics are also worth mentioning. Ref.~\cite{WangNonGauss2020} developed a non-Gaussian chance constraint approach in the case that state constraints can be defined as a polynomial in a random vector. Ref.~\cite{StatContractNG2026} develops a chance-constrained trajectory optimization approach for learning-based motion planners and controllers. Sometimes the non-Gaussian problem can be partially side-stepped, if better working coordinates can be found which render the dynamics less nonlinear, reducing the rate of departure of a distribution from initial Gaussianity (e.g. orbital elements for Keplerian dynamics \cite{junkAdven}, or exponential coordinates for planar robot problems~\cite{BananaIsGaussian2}). However, such coordinates are not always available, especially in complex dynamical systems. Additionally, constraints in trajectory optimization problems are often most conveniently expressed in native physical coordinates. Overall, this area of research extends beyond spaceflight, where new findings are potentially useful for a broad class of control systems for which dynamic nonlinearity, computational constraints, and/or non-Gaussian uncertainty are driving concerns.

In this paper, we develop a method for analytically parameterizing the boundary of a non-Gaussian distribution at a specified confidence (or cumulative probability) level. Finite moments of the distribution are estimated via a deterministic technique and used to inform geometric corrections on baseline geometry driven by covariance and the Gaussian assumption. Modern uncertainty propagation techniques can estimate ``higher order" statistical moments (so-named because moments beyond mean and covariance are often neglected), with e.g. third- and fourth-order moments available at modest computational cost. To compute these moments, we make use of the conjugate unscented transform~\cite{CUT_ACC}. In possession of these moments, the perturbed confidence boundary can be computed and evaluated efficiently.

This paper is structured as follows. We begin with a preliminary discussion on chance-constrained control, then address the standard Gaussian case. We then use some important facts from Gaussian and non-Gaussian statistics to derive the corrected form of a confidence boundary in the case that the first four moments (mean, covariance, skew, and kurtosis) are sufficient descriptors of non-Gaussianity. The paper concludes with use of the confidence boundary in a chance-constrained control algorithm applied to a stochastic spacecraft maneuver targeting problem.

\section{Preliminary}
\subsection{Model}
Consider a discrete nonlinear dynamical system with state vector $\bm{x}\in\mathbb{R}^{n}$, where $\bm{x}_{l}$ specifically denotes the value of the state at time $t_{l}$. The system evolves as
\begin{equation}
\label{StateDyn}
\bm{x}_{l+1} = \bm{\phi}(\bm{x}_{l} + \underline{B}\bm{u}_{l})
\end{equation}
where $\bm{\phi}$ denotes the flow of the dynamics from time $t_{l}$ to $t_{l+1}$, $\bm{u}_{l}\in\mathbb{R}^{m}$ is the control at time $t_{l}$, with associated control matrix $\underline{B}\in\mathbb{R}^{n\times m}$. This discrete dynamical system is a good approximation of the impulsive spacecraft control problems that motivate this work. We consider $\bm{x}_{l}$ as a random vector with mean $\bm{\mu}$, probability density function $p(\bm{x}_{l},t_{l})$, and covariance $P_{l}$. Because $\bm{\phi}$ is nonlinear, even if the initial statistics at $t_{0}$ are Gaussian, at subsequent times this will generally not be the case. For the class of control problems motivating this work, we seek to minimize an objective function $\bm{J}$ subject to path constraints $\bm{g}$. Due to the stochasticity of system~\eqref{StateDyn}, $\bm{g}$ should be satisfied to some probabilistic confidence level. Chance constraints provide a natural way to encode such probabilistic path requirements.
\subsection{Chance Constrained Control}

\begin{definition}[State chance constraint]
Let $\mathcal{S}_{l,j}\subset\mathbb{R}^{n}$ denote the feasible set associated with the $j$\textsuperscript{th} state constraint at time $t_l$, with $j=1,\ldots,N_c$. A state chance constraint is written as
\begin{equation}
\Pr\!\left[\bm{x}_l\in\mathcal{S}_{l,j}\right] \ge 1-\delta_{l,j},
\label{eq:state_chance_constraint}
\end{equation}
where $0<\delta_{l,j}<1$ is the allowable probability of violation.
\end{definition}

\begin{problem}[Chance-constrained control problem]
In this work we solve the simplified problem. Given the dynamics in Eq.~\eqref{StateDyn}, determine a control sequence $\{\bm{u}_l\}_{l=0}^{T-1}$ that minimizes the objective $J$ subject to the dynamics and the state chance constraints:
\begin{subequations}
\begin{align}
\underset{\bm{u}_l}{\text{min}} \ \ J(\bm{u}_{l}) \\ \text{s.t.}\quad & \bm{x}_{l+1}=\bm{\phi}(\bm{x}_l+\underline{B}\bm{u}_l), \\
& \Pr\!\left[\bm{x}_l+1\in\mathcal{S}_{l+1,j}\right] \ge 1-\delta_{l+1,j}.
\label{prob:chance_constrained_control}
\end{align}
\end{subequations}
\end{problem}
In the standard Gaussian setting, and for a half-space constraint of the form $\mathcal{S}_{l,j}=\{\bm{x}:\bm{h}_{l,j}^\top\bm{x}\le g_{l,j}\}$, 
the probabilistic constraint can be converted into a deterministic inequality involving only the mean and covariance~\cite{OguriMcMahoIEEE2019,BlackmoreIEEE2011}.
The present paper is concerned with cases in which the propagated statistics are appreciably non-Gaussian, so this approach
is no longer an adequate description of the relevant confidence geometry. Rather than introducing a mixture-based reformulation as in Ref.~\cite{boone_cdc}, our approach is to approximate the non-Gaussian confidence boundary directly using higher-order moments. This viewpoint motivates the constructions developed in the next section.
\section{Approximating Non-Gaussian Confidence Bounds}
Standard covariance-based confidence boundaries rely on the Gaussian confidence geometry implied by first and second moments alone. In nonlinear stochastic dynamics, state uncertainty distributions can become appreciably perturbed while the covariance remains well defined, so an ellipse inferred only from mean and covariance may misrepresent the true confidence boundary. 
Motivated by this fact, this work derives analytic corrections of the confidence bounds in the case of weakly non-Gaussian statistics, 
leveraging measures of non-Gaussianity from the moments beyond mean and covariance.
\subsection{Gaussian confidence contours}
Let $r \in \mathbb{R}^2$ denote a two-state subset (a ``slice") of a higher-dimensional random vector. Let $\mu \in \mathbb{R}^2$ and $\Sigma \in \mathbb{R}^{2\times 2}$ denote the corresponding slice mean and
covariance, with $\Sigma \succ 0$.

\begin{definition}[Gaussian confidence contour]
\label{def:gaussian_contour}
For a confidence scaling $k(P)>0$ corresponding to a specified confidence level $0<P<1$, the Gaussian confidence region is
\begin{equation}
\mathcal{A}_k \;\triangleq\; \left\{ \bm{r} \in \mathbb{R}^2:\; (\bm{r}-\bm{\mu})^\top \Sigma^{-1} (\bm{r}-\bm{\mu}) \leq k^2 \right\}.
\end{equation}
Its boundary $\partial\mathcal{A}_{k}$ is the Gaussian confidence contour.
\end{definition}

Let $\Sigma = R\Lambda R^\top$ with $\Lambda=\mathrm{diag}(\lambda_1,\lambda_2)$, $\lambda_1\ge \lambda_2>0$,
and $R$ a rotation, 
\begin{equation}
    R =
        \begin{bmatrix}
            \cos\theta & -\sin\theta \\
            \sin\theta & \cos\theta
        \end{bmatrix} .
    \label{eq:R_def}
\end{equation}
\begin{equation}
    \theta = \frac{1}{2} \,\operatorname{atan2}
    \!\big( 2\,\Sigma_{xy},\, \Sigma_{xx}-\Sigma_{yy} \big) .
    \label{eq:theta}
\end{equation}
where Eq.~\eqref{eq:theta} is in terms of components of $\Sigma$. Define the principal semi-axes
\begin{equation}
\overline{a} = k\sqrt{\lambda_1},\qquad \overline{b} = k\sqrt{\lambda_2},
\end{equation}
and the principal coordinates $u,v$ measuring along the principal axis directions. 

\begin{remark}[Standard parameterization]
\label{rem:ellipse_param}
In the case of Gaussian statistics, a convenient parameterization of the \textit{confidence boundary} $\partial \mathcal{A}_k$ at cumulative probability level $P$ is
\begin{equation}
u(t)=\overline{a}\cos t,\qquad v(t)=\overline{b}\sin t,\qquad t\in[0,2\pi),
\label{eq:local_uv2}
\end{equation}
with mapping back to the original sliced coordinates via the contour
\begin{equation}
\bm{r}(t)=\bm{\mu} + R\begin{pmatrix}u(t)\\ v(t)\end{pmatrix}.
\label{MapPrincReg}
\end{equation}
Define normalized principal coordinates $\hat u \triangleq u/\sqrt{\lambda_1}$ and
$\hat v \triangleq v/\sqrt{\lambda_2}$, for which the Gaussian boundary satisfies
$\hat u(t)=k\cos t$ and $\hat v(t)=k\sin t$. 
\end{remark}

\subsection{Moment tensors and Gaussian identities}

Let $\bm{X}\in\mathbb{R}^{d}$ be a random vector with finite central moments through order~4.
Let $\bm{\mu}=\mathbb{E}[\bm{X}]$, $\Sigma=\mathbb{E}[(\bm{X}-\bm{\mu})(\bm{X}-\bm{\mu})^\top]$, and define the central moment tensors
\begin{align}
M^{(3)}_{ijk} &\triangleq \mathbb{E}\!\left[(X_i-\mu_i)(X_j-\mu_j)(X_k-\mu_k)\right],\\
M^{(4)}_{ijkl} &\triangleq \mathbb{E}\!\left[(X_i-\mu_i)(X_j-\mu_j)(X_k-\mu_k)(X_l-\mu_l)\right].
\end{align}
\begin{remark}[Symmetries of moment tensors]
$M^{(3)}$ is invariant under any permutation of $(i,j,k)$ and $M^{(4)}$ is invariant under any permutation
of $(i,j,k,l)$.
\end{remark}

\begin{remark}[Gaussian moment identities]
\label{rem:gaussian_identities}
If $X$ is multivariate Gaussian, then
\begin{align}
M^{(3)}_{ijk} &= 0,\quad \forall i,j,k,\\
M^{(4)}_{ijkl} &= \Sigma_{ij}\Sigma_{kl} + \Sigma_{ik}\Sigma_{jl} + \Sigma_{il}\Sigma_{jk},\quad \forall i,j,k,l.
\end{align}
Equivalently, the 4th-order \textit{cumulant} tensor $\kappa$ vanishes:
\begin{equation}
\kappa_{ijkl} \triangleq M^{(4)}_{ijkl} - (\Sigma_{ij}\Sigma_{kl}+\Sigma_{ik}\Sigma_{jl}+\Sigma_{il}\Sigma_{jk})
\;=\;0.
\end{equation}
These identities are standard consequences of Isserlis' theorem for centered multivariate Gaussian variables~\cite{McCullagh1987}. They provide two tests of Gaussianity generated by the third- and fourth-order statistical moments. Efficient estimates of these moments are available via e.g the algorithm in Ref.~\cite{CUT_ACC}.
\end{remark}

\subsection{Modeling weakly non-Gaussian contours}

We now formalize a viewpoint for corrections due to non-Gaussianity: the Gaussian ellipse remains a useful baseline, but
the true confidence boundary locally departs from it. This inspires a flexible corrective geometric approach whereby expected structure of the confidence boundary is imposed, and the necessary relations of the contour parameters to the statistical moments are derived. In this work, we assume that departures from Gaussianity occur primarily through (i) a bend of the short-axis coordinate as a function of long-axis deviation and (ii) long-axis asymmetry, both to be discussed shortly. These corrections are depicted in Fig.~\ref{f:BananaCorr}, along with the resulting improved constraint satisfaction of the non-Gaussian state distribution at a certain confidence level.
\begin{figure}[h!]
\begin{center}
\includegraphics[]{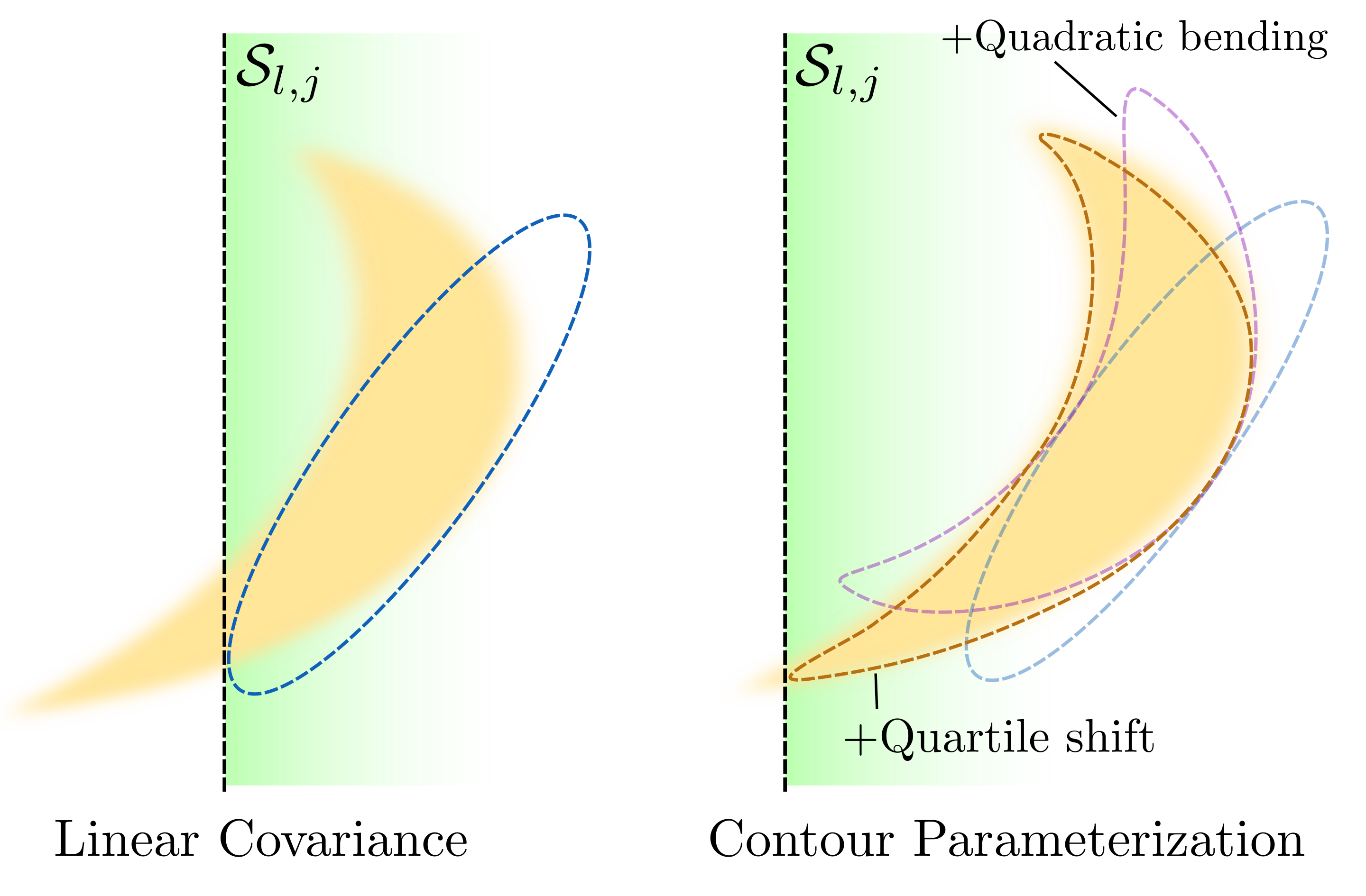}
\caption{Non-Gaussian chance constraints}
\label{f:BananaCorr}
\end{center}
\end{figure}

The form of the needed corrections is problem-driven and can be informed by study of the expected geometry of the propagated distribution. The purpose here is to demonstrate how such corrections can be computed systematically from non-Gaussian moments.

\begin{assumption}[Weakly non-Gaussian regime]
\label{ass:weak_ng}
In the normalized principal coordinates $(\hat u,\hat v)$, the confidence boundary of interest is a smooth deformation of the Gaussian contour $\hat u(t)=k\cos t$, $\hat v(t)=k\sin t$, and the
deformation can be approximated using only the first four central moments of the underlying distribution.
\end{assumption} 
Under the flow of nonlinear dynamics for an initially Gaussian distribution, the central portion of the propagated distribution is often still well-approximated by the Gaussian geometry induced by the mean and covariance. The largest visible departures appear farther from the mean along the extremities of the confidence boundary. We call a distribution satisfying this qualitative behavior ``weakly non-Gaussian". This motivates treating the non-Gaussian contour as a perturbation of the Gaussian ellipse and retaining only first-order corrections in parameters inferred from non-Gaussian moments. 
\subsubsection{First correction: bend (banana centerline curvature)}
\begin{remark}[Banana distributions]\label{rem:gobananas}
It is often observed in applications with nonlinear dynamics, for example orbital mechanics problems \cite{junkAdven}, that the initially ellipsoidal shape of some projection/slice of a distribution bends into a ``banana" shape under the flow of the nonlinear dynamics -- namely, that after affine (linear) stretching effects, the next observable deformation is quadratic in nature.
\end{remark}
\begin{assumption}[Quadratic bending]
\label{ass:bend_ansatz}
Noting that the lowest-order capture of the bending observed in banana distributions is a departure in the short axis $v$ that is quadratic in the long axis $u$ coordinate, for a banana distribution and a given two-state slice $q$, the normalized short-axis coordinate admits the approximate coupling
\begin{equation}
\hat v \approx \beta_q + \alpha_q \hat u^2,
\label{myansatzII}
\end{equation}
where $(\alpha_q,\beta_q)$ are slice-dependent coefficients to be determined from statistical moments, with index $q$ for each possible two-state slice. Higher-order couplings are assumed negligible.
\end{assumption}

Let the whitening map be $\begin{pmatrix}\hat u\\ \hat v\end{pmatrix}=W_q(\bm{r}-\bm{\mu})$, with
$W_q=\Lambda_q^{-1/2}R_q^\top$ for the slice covariance factorization $\Sigma_q=R_q\Lambda_q R_q^\top$. This transformation is convenient for our derivation. 
Define matrix components $\bm{a}=W_q^\top \bm{e}_1$, $\bm{b}=W_q^\top \bm{e}_2$ for $\bm{e}_{1} = (1, 0)^{\top}$, $\bm{e}_{2} = (0, 1)^{\top}$, and $\bm{\delta}=\bm{r}-\bm{\mu}$. If the covariance $\Sigma_{q}$ is propagated from some epoch value $\Sigma_{q,0}$, this can be done in either a linearized or nonlinear fashion. The whitening map could thus in principle be formed from either a linear or nonlinear covariance prediction. We use the linearly predicted covariance, denoted $\Sigma_{q,\text{l}}$. This is consistent with Assumption~\ref{ass:weak_ng} and the perturbative viewpoint adopted here: the Gaussian ellipse induced by affine stretching serves as the zeroth-order organizing geometry, and the banana terms to be derived will provide the first-order non-Gaussian correction. 
\begin{proposition}[Quadratic bending coefficients]
\label{prop:bend_coeffs}
Under Assumption~\ref{ass:bend_ansatz}, the least-squares coefficients minimizing
$\mathbb{E}\!\left[(\hat v-\beta-\alpha \hat u^2)^2\right]$ satisfy $\beta=-\alpha$ and
\begin{equation}
\label{eq:WhatIsAlpha}
\alpha \;=\;
\frac{\sum_{i,j,k} b_i a_j a_k\, M^{(3)}_{q,ijk}}
{\sum_{i,j,k,l} a_i a_j a_k a_l\, M^{(4)}_{q,ijkl} - 1}.
\end{equation}
In the Gaussian case, $\alpha=\beta=0$, by consequence of Remark~\ref{rem:gaussian_identities}. These coefficients provide the best fit of the assumed distribution for the specified ansatz.
\end{proposition}
\begin{proof}
The proof is established in three parts. First, we establish the suitability of the least-squares approach. Then, we apply the conditions of optimality, obtaining $\alpha, \ \beta$ in terms of whitened moments. Then, we relate these results back to the original moments.
\begin{enumerate}
\item We seek to minimize the expected square of the fit error, 
\begin{equation}
(\alpha,\beta)
=
\arg\min_{\alpha,\beta}
\;\mathbb{E}\!\left[\bigl(\hat v - \beta - \alpha\,\hat u^2\bigr)^2\right].
\label{eq:pop_min}
\end{equation}
Under Assumptions~\ref{ass:weak_ng} and~\ref{ass:bend_ansatz}, Eq.~\eqref{eq:pop_min} selects the statistically most accurate quadratic coupling compatible with the adopted ansatz and the retained moment information.
\item Solving Eq.~\eqref{eq:pop_min} by 
applying the first-order conditions of optimality w.r.t. $\alpha$ and $\beta$, we obtain:
\begin{subequations}
\label{eq:pop_normal}
\begin{align}
\mathbb{E}[\hat v] 
&= \beta + \alpha\,\mathbb{E}[\hat u^2],
\label{eq:pop_normal_1}
\\
\mathbb{E}[\hat v\,\hat u^2] 
&= \beta\,\mathbb{E}[\hat u^2] + \alpha\,\mathbb{E}[\hat u^4].
\label{eq:pop_normal_2}
\end{align}
\end{subequations}
These are two linear equations with two unknowns. With some further manipulations, the corrective coefficients are identified as below: 
\begin{subequations}
\label{eq:pop_solution}
\begin{align}
\alpha 
&= 
\frac{
\mathbb{E}[\hat v\,\hat u^2] - \mathbb{E}[\hat u^2]\,\mathbb{E}[\hat v]
}{
\mathbb{E}[\hat u^4] - \mathbb{E}[\hat u^2]^2
},
\label{eq:pop_alpha}
\\[6pt]
\beta
&=
\frac{
\mathbb{E}[\hat u^4]\,\mathbb{E}[\hat v] - \mathbb{E}[\hat u^2]\,\mathbb{E}[\hat v\,\hat u^2]
}{
\mathbb{E}[\hat u^4] - \mathbb{E}[\hat u^2]^2
}.
\label{eq:pop_beta}
\end{align}
\end{subequations}
\item Below we establish the necessary identities mapping transformed moments.
\begin{subequations}
\begin{align}
\label{eq:EfA2}
\mathbb{E}[\hat{u}] = & \ \mathbb{E}[\bm{a}^{\top}\bm{\delta}] =  \bm{a}^{\top}\left(\mathbb{E}[\bm{r}] - \mathbb{E}[\bm{\mu}_{q}]\right) = 0 \\
\mathbb{E}[\hat{v}] = & \ \mathbb{E}[\bm{b}^{\top}\bm{\delta}] =  \bm{b}^{\top}\left(\mathbb{E}[\bm{r}] - \mathbb{E}[\bm{\mu}_{q}]\right) = 0
\end{align}
\end{subequations}
Index notation is adopted for convenience, e.g. 
\begin{subequations}
\begin{align}
\label{eq:EfA3}
\begin{split}
\hat{u}^{2} = & \ \left(\bm{a}^{\top}\bm{\delta}\right)^{2} = \left(\sum_{j}a_{j}\delta_{j}\right)\left(\sum_{k}a_{k}\delta_{k}\right) \\ = & \  \sum_{j}\sum_{k}a_{j}a_{k}\delta_{j}\delta_{k}
\end{split} \\
\hat{v}\hat{u}^{2} = & \ \sum_{i}b_{i}\delta_{i}\cdot\hat{u}^{2} = \sum_{i}\sum_{j}\sum_{k}b_{i}a_{j}a_{k}\delta_{i}\delta_{j}\delta_{k} 
\end{align}
\end{subequations}
From such expressions it is simple to compute the necessary expected values, e.g.:
\begin{equation}
\label{eq:EfA4}
\begin{split}
\mathbb{E}[\hat{v}\hat{u}^{2}] = & \ \sum_{i}\sum_{j}\sum_{k}b_{i}a_{j}a_{k}\mathbb{E}[\delta_{i}\delta_{j}\delta_{k}] \\ = & \ \sum_{i,j,k} b_{i}a_{j}a_{k}M^{(3)}_{q,ijk}
\end{split}
\end{equation}
Establish one more identity by noting $\Sigma_{q} = R_{q}\Lambda_{q}R_{q}^{\top}$:
\begin{equation}
\label{eq:EfA4}
\begin{split}
\mathbb{E}[\hat{u}^{2}] = & \ \bm{a}^{\top}\Sigma_{q}\bm{a} \\ = & \  \begin{bmatrix}1 & 0\end{bmatrix}\Lambda_{q}^{-1/2}R_{q}^{\top}\Sigma_{q}R_{q}\left(\Lambda_{q}^{-1/2}\right)^{\top}\begin{bmatrix} 1 \\ 0 \end{bmatrix} \\ = & \ 1
\end{split}
\end{equation}
The final result for Eqs.~\eqref{eq:pop_alpha}-\eqref{eq:pop_beta} is thus obtained as Eq.~\eqref{eq:WhatIsAlpha}, as a function of skew and kurtosis.
\end{enumerate}
\end{proof}
\subsubsection{Second correction: long-axis asymmetry}
\begin{remark}[Quartile Asymmetries]\label{rem:gobananas2}
Another commonly observed effect in non-Gaussian statistics is that the symmetry of the confidence bounds along a direction $\bm{s}$ and its opposite $-\bm{s}$ are broken, due to a breakdown in quartile symmetries. We note that this effect is naturally more pronounced along directions with larger variance. 
\end{remark}
\begin{definition}[Cornish-Fisher Expansion]
\label{def:CF1}
The Cornish-Fisher expansion is a well-known asymptotic correction from univariate statistical analysis. It provides an approximation of the quantiles of a non-Gaussian distribution based on its cumulants \cite{HMF}. 
Below we provide just the first-order term of the expansion: 
\begin{equation}
\label{CFish1}
x(p) \approx \mu + \sigma \left(z + \frac{\gamma_{1}}{6}(z^{2}-1) \right)
\end{equation}
where $x$ follows a slightly non-Gaussian univariate distribution, with mean $\mu$ and standard deviation $\sigma$. Furthermore $z=\Phi^{-1}(P)$ where $\Phi$ is the cumulative distribution function of the standard normal distribution, i.e. $\Phi(3)\approx0.9987$. Lastly, $\gamma_{1}$ can be expressed in terms of skew and standard deviation:
\begin{equation}
\label{CFish2}
\gamma_{1} = \frac{\mu_{3}}{\sigma^{3}}, \ \ \ \mu_{n} = \mathbb{E}[(x-\mu)^{n}]
\end{equation}
\end{definition}
\begin{assumption}[Cornish-Fisher shift]
\label{ass:CF1}
Because quartile asymmetries are most pronounced along high-variance directions, we seek a correction to the long axis of the distribution along $\hat{u}$, and not along the short axis $\hat{v}$. This is similar to the logic by which we prioritized bending of the form $\hat{v} \propto \hat{u}^{2}$ but ignored the (assumed sub-dominant) analogous bending term $\hat{u} \propto \hat{v}^{2}$. 
We assume an additive $t$-periodic correction to the parameterization $\hat{u}(t) = k\cos{t}$ which obeys the following properties for $c(k)=\frac{\gamma_{1}}{6}(k^{2}-1)$:
\begin{subequations}
\label{CFish2}
\begin{align}
\delta\hat{u}(0) = \delta\hat{u}(\pi) = & \ c(k) \\
\delta\hat{u}(\frac{\pi}{2}) = \delta\hat{u}(\frac{3\pi}{2}) = & \ 0 \\
\delta\hat{u}'(0) = \delta\hat{u}'(\pi) = & \ 0 \\
\delta\hat{u}'(\frac{\pi}{2}) = \delta\hat{u}'(\frac{3\pi}{2}) = & \ 0 
\end{align}
\end{subequations}
where $( \ )' = \frac{\text{d}}{\text d t}( \ )$. The first two conditions state that along the direction $\hat{u}$, we expect to recover the univariate correction: The $k\sigma$ boundaries shift to $-k + c(k)$ for $\hat{u}<0$ and $k+c(k)$ for $\hat{u}>0$. The second two enforce no net change when $\hat{u} = 0$. 
The last four conditions are for regularity -- to avoid unphysical directional biases. Noting that the univariate Cornish-Fisher expansion provides no information about how the contour changes except purely along the long axis of the distribution, we furthermore require evenness of $\delta\hat{u}(t)$ about $t=0$ and $t=\pi$, as any other choice induces off-axis asymmetries that cannot be justified. The class of all functions satisfying these conditions can be expressed via the cosine series:
\begin{equation}
\label{duCS}
\delta\hat{u}_{q}(t) = \eta_{0,q} + \sum_{n\geq1}\eta_{n,q}\cos{(nt)}
\end{equation}
We noted previously that the first four moments of a non-Gaussian distribution are not enough to uniquely determine the new chance contour. In general a non-Gaussian distribution is characterized by infinitely many moments of descending importance. Thus we propose the simplest justifiable correction based on known information and constraints of Eq.~\eqref{CFish2}, arriving at our second ansatz:
\begin{equation}
\label{myansatzB}
\delta\hat{u} = \eta_{0} + \eta_{2}\cos{2t}
\end{equation}
\end{assumption}
\begin{proposition}[Long axis shift coefficients]
\label{prop:CF_coeffs}
Under Assumption~\ref{ass:CF1}, the desired coefficients of Eq.~\eqref{myansatzB} are
\begin{subequations}
\begin{align}
\eta_{0} = & \ \frac{1}{2}c(k) \\
\eta_{2} = & \ \eta_{0},
\end{align}
\end{subequations}
where $c(k)$ is given in terms of skew as
\begin{equation}
\label{ckSkew}
c(k)
= \frac{k^2 - 1}{6}
  \sum_{i,j,k}
  a_i\,a_j\,a_k\,M^{(3)}_{q,ijk}
\end{equation}
\end{proposition}
\begin{proof}
The proof is established in the same manner as Proposition~\ref{prop:bend_coeffs}: setup and solve of constraint equations to be satisfied by the ansatz. Namely, the coefficients of Eq.~\eqref{myansatzB} must satisfy Eqs.~\eqref{CFish2}, then moment identities like Eqs.~\eqref{eq:EfA3} establish the result in terms of computed (non-whitened) higher-order moments. 
\end{proof}
\subsubsection{Final Result}
At this point, the corrective quantities have been expressed directly in terms of the propagated third- and fourth-order moments. The remaining step is to embed these quantities into a practical contour parameterization that reduces to the Gaussian ellipse when the non-Gaussian indicators vanish, while preserving the intended bending and long-axis asymmetry to first order. This yields a closed-form surrogate for the confidence boundary that is simple to evaluate and therefore suitable for algorithmic use inside trajectory optimization routines.
\begin{definition}[Banana contour parameterization]
\label{def:bent_param}
Incorporating the quadratic bending term and the uniaxial quartile shift, a moment-corrected contour in principal coordinates can be written. Revisiting Eqs.~\eqref{eq:local_uv2}, ~\eqref{myansatzII}, and~\eqref{myansatzB}, we note the zeroth-order solution can be written as $\hat{u} = k\cos{t}$, $\hat{v} = k\sin{t}$. We add the first-order corrections, noting that in the event that $\alpha=0$ and $c(k)=0$, the resulting $\hat{u}$ and $\hat{v}$ are unmodified from the Gaussian case. 
After rescaling (undoing the normalization), we make use of a double-angle identity to the correction in $u$, highlighting the cosine-squared commonality of the corrective terms in $u$ and $v$:
\begin{align}
u(t) &= \overline{a}\cos t + c(k)\sqrt{\lambda_{1}}\cos^{2}{t},\qquad \overline{a}=k\sqrt{\lambda_1}, \label{eq:u_banana}\\
v(t) &= \overline{b}\sin t + \alpha \sqrt{\lambda_2}\left(k^2\cos^2 t - 1\right),\qquad \overline{b}=k\sqrt{\lambda_2}, \label{eq:v_banana}
\end{align}
This is mapped back to $\bm{r}$ via Eq.~\eqref{MapPrincReg} (unchanged from the Gaussian case). This is a first-order correction, linear in perturbative parameters, which thus can be applied independently slice-by-slice (e.g. different $q$) to accommodate multiple bendings if needed, with each slice referencing different parts of the skew and kurtosis tensors. 
\end{definition}

\section{Numerical Example}
We demonstrate the proposed banana contour model on the impulsive maneuver targeting scenario from Ref.~\cite{boone_cdc}, with a spacecraft orbiting an asteroid modeled as a point mass.

\subsection{Dynamics model: two-body problem}
The dynamics of a spacecraft orbiting a single body can be approximated by assuming the central body is a point mass. The equations of motion for the two-body problem are

\begin{align}
& \ddot{x} = -\frac{\mu x}{r^3} &
& \ddot{y} = -\frac{\mu y}{r^3} &
& \ddot{z} = -\frac{\mu z}{r^3}
\end{align}
where the state vector $\bm{x} = \begin{bmatrix}x & y & z & \dot{x} & \dot{y} & \dot{z} \end{bmatrix}^T$ contains 3 position $(x,y,z)$ and 3 velocity $(\dot{x}, \dot{y}, \dot{z})$ components. The distance $r$ is the distance from the spacecraft to the central body (i.e., $r = \sqrt{x^2 + y^2 + z^2}$), and $\mu$ is the gravitational parameter for the central body.

\subsection{Scenario}
Consider the scenario from Ref.~\cite{boone_cdc} with a spacecraft orbiting a small asteroid ($\mu = 5.2 \: \textrm{m}^3/\textrm{s}^2$) in a circular terminator orbit. The initial state values for the spacecraft's nominal trajectory are $\bm{x} = \begin{bmatrix} -1000 & 0 & 0 & 0 & 0 & -7.211\times 10^{-2} \end{bmatrix}^T \: (\textrm{m}, \textrm{m}/\textrm{s})$. We consider the case where a control $\bm{u}_0$ is computed at time $t_0$ to place the spacecraft on a reconnaissance orbit that approaches the surface of the asteroid (see Fig.~\ref{f:asteroid_scenario_shorter}). We want to plan the maneuver to ensure that the spacecraft will not arrive too close to the asteroid at a target time $t_1$, given some initial state uncertainty. This may be operationally desirable for this type of mission, since state and navigation uncertainties are relatively large, and any unforeseen close approach may result in mission failure. 

The following initial state uncertainty values are applied at time $t_0$ to the spacecraft system, where $\epsilon\ll1$ (a case with initial position-dominant dispersions, as escape velocity for this problem is only $\sim$0.1 m/s): 
\begin{equation}
\bm{\sigma}_0 = \begin{bmatrix} 1.0 & 1.0 & 1.0 & \epsilon & \epsilon & \epsilon \end{bmatrix}^T (\textrm{m}, \textrm{m}/\textrm{s})
\end{equation}

We assume that $P_0 = \textrm{diag}([\bm{\sigma_0^2}])$ and that these initial state uncertainties are Gaussian. In order to introduce significant nonlinearities into the trajectory, the final time $t_1$ is set after 1.5 revolutions of the reconnaissance orbit from the maneuver location. This corresponds to roughly 23.6 hours. At the target time, the uncertainty distribution becomes decidedly non-Gaussian when expressed in Cartesian coordinates.


\begin{figure}
\begin{center}
\includegraphics[scale=0.6]{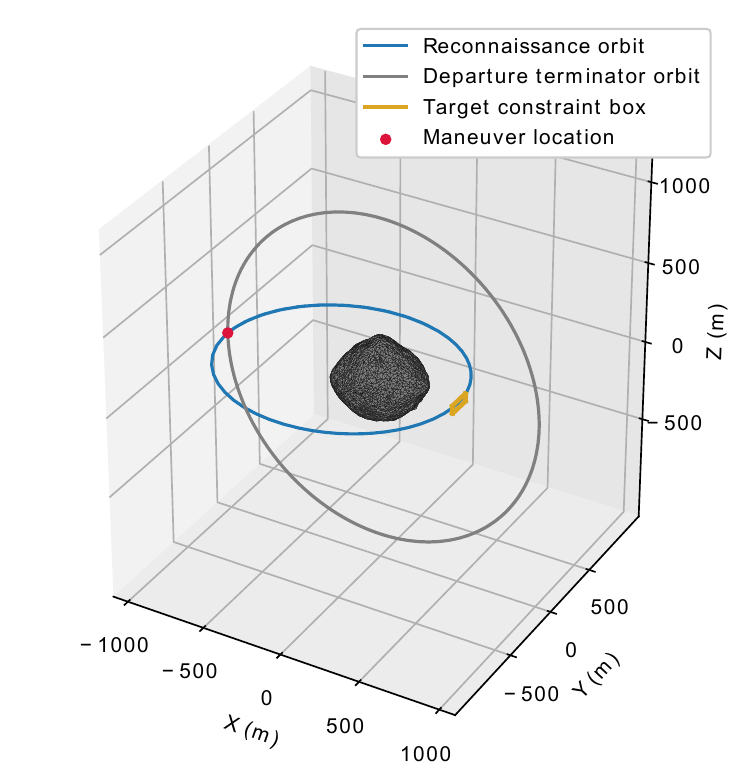}
\caption{Asteroid maneuver targeting scenario}
\label{f:asteroid_scenario_shorter}
\end{center}
\end{figure}

\subsection{Optimization with chance constraints}
We now demonstrate how the banana contour parameterization can be used to apply chance constraints to a non-Gaussian distribution. For this scenario, the target state constraint is defined as a box. This can be formulated using six chance constraints defining the boundaries of the target box in position space. The risk threshold for each chance constraint is set to $\Delta_1 = 0.01$; in other words, we want each chance constraint to be satisfied by $99\%$ of the state distribution at the target time. The six chance constraints can be expressed in the asteroid-centric inertial Cartesian frame:


\begin{align}
& \textrm{Pr} [x_1 \geq 495 \textrm{ m}] \geq 1 - \Delta_1 \\
& \textrm{Pr} [x_1 \leq 505 \textrm{ m}] \geq 1 - \Delta_1 \\
& \textrm{Pr} [y_1 \geq -80 \textrm{ m}] \geq 1 - \Delta_1 \\
& \textrm{Pr} [y_1 \leq 80 \textrm{ m}] \geq 1 - \Delta_1 \\
& \textrm{Pr} [z_1 \geq -25 \textrm{ m}] \geq 1 - \Delta_1 \\
& \textrm{Pr} [z_1 \leq 25 \textrm{ m}] \geq 1 - \Delta_1
\label{eq:cc_box}
\end{align}
This corresponds to a box roughly $250 \textrm{ m}$ altitude above the asteroid surface (see Fig.~\ref{f:asteroid_scenario_shorter}).

We seek to minimize the control norm at time $t_0$, so the objective function is defined as $J = \|\bm{u}_0\|$. An initial guess for the control vector is $\bm{u}_{0,\textrm{guess}} = \begin{bmatrix} 0.0 & 0.058 & 0.072111 \end{bmatrix} \textrm{m/s}$. In order to enforce the chance constraints, points are chosen from the constraint boundary. The optimization algorithm is set up to enforce that all selected points on the boundary satisfy all constraints. The optimization is performed using the SLSQP algorithm available with SciPy's \emph{minimize} function. In this approach, the optimal control parameters are computed by a sequential optimization process, with the confidence boundary recomputed for each iterate.

First, the optimization was performed using the linear covariance propagation, which computes and enforces satisfaction of an ellipse-shaped confidence boundary. Next, the optimization was performed using the contour defined by Eqs.~\ref{eq:u_banana} and~\ref{eq:v_banana}. In both cases, the optimization is able to converge on an impulsive maneuver that satisfies the desired boundary. The resulting covariance ellipses at time $t_1$ are shown in Fig.~\ref{f:linear_plot} for the linear covariance propagation, and in Fig.~\ref{f:banana_plot} for the banana contour parameterization. 

\begin{figure}
\begin{center}
\includegraphics[scale=1.2]{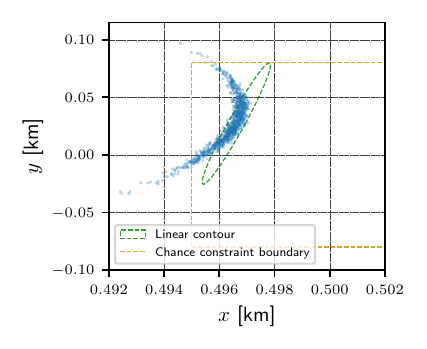}
\caption{Monte Carlo results for linear covariance ellipse chance constraints}
\label{f:linear_plot}
\end{center}
\end{figure}

\begin{figure}
\begin{center}
\includegraphics[scale=1.2]{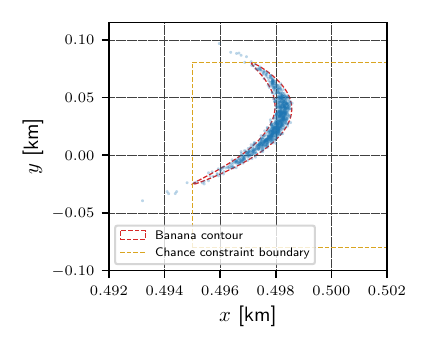}
\caption{Monte Carlo results for banana contour chance constraints}
\label{f:banana_plot}
\end{center}
\end{figure}

To evaluate the performance of the computed optimal control with the true statistical distribution, 5000 state perturbations were sampled from the initial state distribution, and the optimal controls from the linear covariance propagation and the banana contour parameterization were applied. 
These state vectors were propagated to time $t_1$, and the final positions are shown on Figs.~\ref{f:linear_plot} and~\ref{f:banana_plot}. We emphasize that the Monte Carlo samples are not used at all in the computation of the contour in Fig.~\ref{f:banana_plot}. Comparing the two figures, the banana contour method clearly conforms much better to the true physical shape of the state distribution at time $t_1$. As a result, the control computed using this chance constraint method results in improved performance in the Monte Carlo trial. Specifically, 92.2\% of the samples from the linear covariance simulation were found to satisfy all of the chance constraints, whereas 98.3\% of the samples from the non-Gaussian method were found to satisfy all constraints. 
Note that each individual chance constraint is enforced separately, and each individual chance constraint is satisfied to the desired 99.0\% probability. 

Overall, the new contour method is found to be more numerically efficient than the method developed in Ref.~\cite{boone_cdc} which uses a Gaussian mixture model to model the non-Gaussian distribution and an iterative risk allocation algorithm to optimally assign risk to the different mixands. By modeling the non-Gaussian distribution using a single contour boundary, this method obviates the need for an expensive iterative method. 

\section{Conclusions}
This paper presents a methodology for approximating non-Gaussian chance constraints via use of an analytic parameterization of the boundary of the distribution at a specified confidence level. The first four moments of the distribution are estimated via a deterministic technique and used to inform geometric corrections on the baseline Gaussian geometry. We apply this to so-called ``banana-shaped" distributions commonly observed in orbital mechanics problems. The approach is then applied to enforce chance constraints on a non-Gaussian distribution in a constrained optimal spacecraft maneuver targeting scenario. The algorithm accurately satisfies the desired chance constraint boundary, outperforming a traditional linear covariance-based strategy.

Future work will focus on improving the efficiency of the algorithm which enforces the geometric conditions of the chance constraint. The current algorithm applies a given path constraint (e.g. a half-plane constraint) jointly at many sampled points on the analytic contour. A better approach is to identify the worst-violating point and compute the correction based on that single constraint alone. This is more difficult due to the potential for the worst-violating point to jump back and forth between distinct ``lobes" of the non-convex contour of the non-Gaussian shape. In addition, the manner in which cumulative probability is treated in this methodology could perhaps be refined. We will also develop extensions of the geometric technique applied here, and investigate other methods of estimating necessary nonlinear statistical moments. Lastly, future work will develop more sophisticated examples with non-Gaussian chance constraints at multiple distinct times, and more complex dynamics. Ongoing numerical experiments in more non-Gaussian scenarios with more dynamic complexity, including a high-fidelity cislunar study, suggest that the present approach is effective beyond the simple asteroid-targeting example considered here.
\bibliographystyle{IEEEtran}
\bibliography{references}

\end{document}